\documentclass[12pt]{amsart}
\usepackage{amsfonts,graphicx,amssymb,amscd,amsmath,enumerate,verbatim,calc}

\def\NZQ{\mathbb}               

\def\QQ{{\NZQ Q}}
\def\ZZ{{\NZQ Z}}
\def\RR{{\NZQ R}}

\def\ab{{\bold a}}
\def\hb{{\bold h}}
\def\eb{{\bold e}}

\def\eb{{\bold e}}
\def\opn#1#2{\def#1{\operatorname{#2}}} 
\opn\aff{aff} \opn\con{conv} \opn\relint{relint} \opn\st{st}
\opn\lk{lk} \opn\cn{cn} \opn\core{core} \opn\vol{vol}
\opn\link{link} \opn\star{star}

\def\Hc{{\mathcal H}}
\def\Sc{{\mathcal S}}
\def\Cc{{\mathcal C}}

\def\Pc{{\mathcal P}}

\def\Qc{{\mathcal Q}}
\newtheorem{Theorem}{Theorem}[section]
\newtheorem{Lemma}[Theorem]{Lemma}
\newtheorem{Corollary}[Theorem]{Corollary}
\newtheorem{Proposition}[Theorem]{Proposition}

\newtheorem{Question}[Theorem]{Question}
\theoremstyle{definition}
\newtheorem{Remark}[Theorem]{Remark}

\newtheorem{Example}[Theorem]{Example}

\newtheorem{Definition}[Theorem]{Definition}
\let\epsilon\varepsilon
\let\phi=\varphi
\let\kappa=\varkappa
\begin{document}
\title{Integer decomposition property of dilated polytopes}
\author{David A. Cox, Christian Haase, Takayuki Hibi and Akihiro Higashitani}
\address{David A. Cox, 
Department of Mathematics, Amherst College, 
Amherst, MA 01002-5000, USA}
\email{dac@math.amherst.edu}
\address{Christian Haase, Mathematics,
Goethe-Universit\"at Frankfurt/Main, Germany}
\email{haase@math.uni-frankfurt.de}
\address{Takayuki Hibi,
Department of Pure and Applied Mathematics,
Graduate School of Information Science and Technology,pe
Osaka University,
Toyonaka, Osaka 560-0043, Japan}
\email{hibi@math.sci.osaka-u.ac.jp}
\address{Akihiro Higashitani,
Department of Pure and Applied Mathematics,
Graduate School of Information Science and Technology,
Osaka University,
Toyonaka, Osaka 560-0043, Japan}
\email{a-higashitani@cr.math.sci.osaka-u.ac.jp}
\begin{abstract} 
Let $\Pc \subset \RR^N$ be an integral convex polytope of dimension $d$ 
and write $k \Pc$, where $k = 1, 2, \ldots$\,, for
dilations of $\Pc$.
We say that $\Pc$ 
possesses the integer decomposition property if, for any integer $k =
1, 2, \ldots$ and for any $\alpha \in k \Pc \cap \ZZ^{N}$, there exist
$\alpha_{1}, \ldots, \alpha_k$ belonging to $\Pc \cap \ZZ^N$
such that $\alpha = \alpha_{1} + \cdots + \alpha_k$. 
A fundamental question is to determine the integers $k > 0$ for which 
the dilated polytope $k\Pc$ possesses the integer decomposition property.
In the present paper, combinatorial invariants related to the integer
decomposition property of dilated polytopes will be proposed and studied.
\end{abstract}

\maketitle

\section*{Introduction}
Integral convex polytopes have been studied from the viewpoints of
commutative algebra and algebraic geometry together with enumerative
combinatorics, combinatorial optimization and statistics.
Recall that an {\em integral} convex polytope is a convex polytope all
of whose vertices have integer coordinates.

There is an entire network \cite[p.2313]{MFO07} of combinatorial and
algebraic properties of integral polytopes for which it is interesting
to study the behavior under dilations \cite{BGcover,BGT,KKMS}.
In the present paper, we are especially interested in the integer
decomposition property of integral convex polytopes, which is
particularly important in the theory and application of integer
programming \cite[\S22.10]{Sch}.

{\bf (0.1)}
We say that an integral convex polytope 
$\Pc \subset \RR^{N}$ 
possesses the {\em integer decomposition property} or (IDP) for short
if, for every integer $k = 1, 2, \ldots$ and for all $\alpha \in k \Pc
\cap \ZZ^{N}$, there exist $\alpha_{1}, \ldots, \alpha_k$ belonging to
$\Pc \cap \ZZ^N$ such that $\alpha = \alpha_{1} + \cdots + \alpha_k$.
Here, $k \Pc = \{ k \alpha : \alpha \in \Pc \} \subset \RR^{N}$ is the
{\em dilated polytope}.

Furthermore, we say that an integral convex polytope $\Pc \subset \RR^{N}$
is {\em very ample} if, for any sufficiently large integer $k \gg 0$ 
and for any $\alpha \in k\Pc \cap \ZZ^N$, there exist $\alpha_1, \ldots, \alpha_k$ 
belonging to $\Pc \cap \ZZ^N$ such that $\alpha = \alpha_1 + \cdots + \alpha_k$. 
Thus in particular if $\Pc$
possesses (IDP), then $\Pc$ is very ample.


{\bf (0.2)}
Let $\Cc \subset \RR^{N}$ be a {\em pointed, rational and polyhedral cone}
generated by rational vectors 
$\ab_{1}, \ldots, \ab_{m} \in \QQ^{N}$.  Thus
\[
\Cc = \RR_{\geq 0}\{\ab_1, \ldots, \ab_m\}=
\{\lambda_{1} \ab_{1} + \cdots + \lambda_{m} \ab_{m} :
\lambda_{1}, \ldots, \lambda_{m} \in \RR_{\geq 0}\}
\]  
such that $\{ 0 \}$ is the largest linear subspace contained in $\Cc$.
A finite set of integer vectors
$\{ \hb_{1}, \ldots, \hb_{s} \} \subset \ZZ^{N}$ is called a {\em Hilbert basis} of $\Cc$ if
\[
\Cc \cap \ZZ^{N} = \ZZ_{\geq 0}\{\hb_1,\ldots,\hb_s\} :=
\{\lambda_{1} \hb_{1} + \cdots + \lambda_{s} \hb_{s} :
\lambda_{1}, \ldots, \lambda_{s} \in \ZZ_{\geq 0}\}.
\]
A Hilbert basis exists~\cite{Gordan} and a minimal Hilbert basis is
unique~\cite{vanDerCorput}.
Let $\Hc(\Cc)$ denote {\em the} minimal Hilbert basis of $\Cc$.
In terms of Hilbert bases, an integral convex polytope $\Pc \subset
\RR^{N}$ is very ample if and only if for all vertices $\alpha$ of
$\Pc$ the set $(\Pc-\alpha) \cap \ZZ^N$ is 
a Hilbert basis for the cone it generates (cf.~\cite[Ex.~2.23]{BG}).

We will write $\widetilde{\Pc} \subset \RR^{N+1}$ for the integral
convex polytope
$\{ (\alpha, 1) \in \RR^{N+1} : \alpha \in \Pc \}$.
Let $\Cc(\Pc) \subset \RR^{N+1}$ denote the pointed rational
polyhedral cone generated by the vertices of $\widetilde{\Pc}$.
The {\em degree} of $(\alpha,n) \in \Cc(\Pc) \cap \ZZ^{N+1}$ is   
$\deg (\alpha,n) = n$.

{\bf (0.3)}
A simplex $\Sc \subset \RR^{N}$ is called
{\em empty} if 
$\Sc \cap \ZZ^{N}$
is the set of vertices of $\Sc$.
Given an empty simplex $\Sc \subset \RR^N$, we define
the finite subset $\text{Box}(\Sc) \subset \Cc(\Sc)$ as follows:
\[
\text{Box}(\Sc)
= \left\{\sum_{v_i \in \Sc \cap \ZZ^N} r_i (v_i,1) \in \ZZ^{N+1} \, : \,
0 \leq r_i < 1 \right\}.
\] 
We have $\Hc(\Cc(\Sc)) \subseteq \text{Box}(\Sc) \cup (\widetilde{\Sc} \cap \ZZ^{N+1})$.

Let, in general, $\Pc \subset \RR^N$ be an integral convex polytope of dimension $d$.
We then define
$\text{Box}(\Pc)$ by setting
\[
\text{Box}(\Pc)=\bigcup_{\Sc} \text{Box}(\Sc) \setminus \ZZ_{\geq 0}(\widetilde{\Pc} \cap \ZZ^{N+1}),
\]
where $\Sc$ runs over all empty simplices of dimension $d$ with $\Sc \subset \Pc$. 
Each element belonging to $\text{Box}(\Pc)$ is called a {\em hole} of $\Pc$.
We have $\Hc(\Cc(\Pc)) \subseteq \bigcup_{\Sc} \Hc(\Cc(\Sc)) \subseteq
\text{Box}(\Pc) \cup (\widetilde{\Pc} \cap \ZZ^{N+1})$.

{\bf (0.4)}
Let $\delta(\Pc) = (\delta_{0}, \delta_{1}, \ldots, \delta_{d})$
denote the {\em $\delta$-vector} (\cite[Chapter XI]{HibiRedBook}) 
of an integral convex polytope $\Pc \subset \RR^{N}$ of dimension $d$.
We write $k_{0}$ for the smallest integer $k > 0$
for which $k(\Pc \setminus \partial \Pc) \cap \ZZ^{d} \neq \emptyset$
and $i_{0}$ for the largest integer $i$ for which $\delta_{i} \neq 0$.
It is known \cite[Theorem 4.5]{BR} that 
\[
k_{0} = (d + 1) - i_0.
\]

\begin{Definition}
Given an integral convex polytope $\Pc \subset \RR^{N}$,
we introduce the invariants 
\[
\mu_\text{va}(\Pc), \, \, \mu_\text{midp}(\Pc), \, \, \mu_\text{idp}(\Pc), \, \, 
\mu_\text{Hilb}(\Pc), \, \, \mu_\text{hole}(\Pc), \, \, \mu_\text{Ehr}(\Pc),
\]
each of which is defined in what follows:
\begin{itemize}
\item
$\mu_\text{va}(\Pc)$ is the smallest integer $k > 0$ for which 
the dilated polytope $k\Pc$ is very ample; 
\item
$\mu_\text{midp}(\Pc)$ is the smallest integer $k > 0$ for which 
the dilated polytope $k\Pc$ possesses (IDP); 
\item
$\mu_\text{idp}(\Pc)$ is the smallest integer $k > 0$ for which 
the dilated polytopes $n\Pc$ possess 
(IDP) for all $n \geq k$;
\item
$\mu_\text{Hilb}(\Pc)$ is the maximal degree of elements belonging to 
the minimal Hilbert basis $\Hc(\Cc(\Pc))$ of $\Cc(\Pc)$;
\item
$\mu_\text{hole}(\Pc)$ is the maximal degree of elements belonging to $\text{Box}(\Pc)$,
with the convention that $\mu_\text{hole}(\Pc) = 1$ if $\text{Box}(\Pc) = \emptyset$; 
\item
$\mu_\text{Ehr}(\Pc)$ is the largest integer $i > 0$ for which $\delta_i \not=0$, 
where $(\delta_0,\delta_1,\ldots,\delta_d)$ denotes the $\delta$-vector of $\Pc$. 
\end{itemize}
\end{Definition}

The final goal of our research project is to find a combinatorial characterization 
of the sequences
\[
\mu(\Pc)
= (\mu_\text{va}(\Pc),\mu_\text{midp}(\Pc),\mu_\text{idp}(\Pc),
\mu_\text{Hilb}(\Pc),\mu_\text{hole}(\Pc),\mu_\text{Ehr}(\Pc))
\]
arising from integral convex polytopes $\Pc$ of dimension $d$.

First of all, in Section $1$, we give basic inequalities which the
above sequence $\mu(\Pc)$ satisfies.
More precisely, Theorem \ref{invariant} says that
\[
1 \leq \mu_\text{va}(\Pc) \leq \mu_\text{midp}(\Pc) \leq \mu_\text{idp}(\Pc)
\leq \mu_\text{hole}(\Pc) \leq \mu_\text{Ehr}(\Pc) \leq d;
\]
\[
1 \leq \mu_\text{va}(\Pc) \leq 
\mu_\text{Hilb}(\Pc) \leq \mu_\text{hole}(\Pc) \leq d - 1.
\]

Various examples will be supplied in Section $2$.
In Theorem \ref{Berkeley}, we prove that,
given integers $d \geq 3$ and $2 \leq j \leq d - 1$, 
there exists an empty simplex $\Pc$ of dimension $d$ with
$\mu(\Pc) = (j,j,j,j,j,j)$.
In Theorem \ref{Boston},
we prove that,
given an integer $d \geq 4$, 
there exists an integral convex polytope $\Pc$ of dimension $d$ 
such that $\mu_\text{Hilb}(\Pc) = d - 1$ and
$\mu_\text{midp}(\Pc) = \mu_\text{idp}(\Pc) = d - 2$.
In addition, we give examples of integral convex polytopes
$\Pc$, $\Pc'$ and $\Pc''$ with
\begin{itemize}
\item
$\mu_\text{va}(\Pc) < \mu_\text{midp}(\Pc), \, \,  
\mu_\text{va}(\Pc) < \mu_\text{Hilb}(\Pc)$;
\item
$\mu_\text{midp}(\Pc') < \mu_\text{idp}(\Pc')
< \mu_\text{hole}(\Pc'), \, \, 
\mu_\text{Hilb}(\Pc') < \mu_\text{hole}(\Pc')$;
\item
$\mu_\text{hole}(\Pc'') < \mu_\text{Ehr}(\Pc'')$.
\end{itemize}

Moreover, in Section $3$, more detailed relations between 
$\mu_\text{midp}(\Pc)$ and $\mu_\text{idp}(\Pc)$ will be discussed
(Theorem \ref{teiri}).
Finally, we compute in Section $4$ the invariants of edge polytopes
arising from finite graphs.

\section{Invariants related to dilated polytopes}

In this section, we discuss the sequence $\mu(\Pc)$ of the invariants 
of an integral convex polytopes $\Pc$ of dimension $d$ related to (IDP). 
More precisely, we prove the following.

\begin{Theorem}\label{invariant}
For the invariants appearing in the introduction 
of an integral convex polytope $\Pc \subset \RR^N$ of dimension $d$, 
the following inequalities hold: 
\begin{equation}\label{diagram}
\begin{aligned}
&1 \leq \mu_\text{{\em va}}(\Pc) \leq \mu_\text{{\em midp}}(\Pc) \leq 
\mu_\text{{\em idp}}(\Pc) \leq \mu_\text{{\em hole}}(\Pc) \leq \mu_\text{{\em Ehr}}(\Pc) \leq d \\
&\;\;\; \text{{\em and }} \;\; 
1 \leq \mu_\text{{\em va}}(\Pc) \leq \mu_\text{{\em Hilb}}(\Pc) \leq \mu_\text{{\em hole}}(\Pc) \leq d-1. 
\end{aligned}
\end{equation}
\end{Theorem}
\begin{proof}
The inequalities 
$1 \leq \mu_\text{va}(\Pc) \leq \mu_\text{midp}(\Pc) \leq \mu_\text{idp}(\Pc)$ and $\mu_\text{Ehr}(\Pc) \leq d$ 
are clear from their definitions. 
On the other hand, since the assertions are obvious if $\Pc$ has (IDP), 
we assume that $\text{Box}(\Pc)$ is not empty. \\
$\bullet$ \underline{$\mu_\text{idp}(\Pc) \leq \mu_\text{hole}(\Pc)$:}
This inequality is proved, though not stated, in
\cite{LiuTrotterZiegler}. We give the proof for the sake of completeness.

It follows from Gordan's Lemma \cite[Theorem 16.4]{Sch} and
Carath\'eodory's Theorem \cite[Corollary 7.1i]{Sch} that $\Cc(\Pc)
\cap \ZZ^{N+1}$ consists of the elements of
$$ 
\{ \alpha  + x \ : \ \alpha \in \text{Box}(\Pc) \cup \{0\}, \ x \in \ZZ_{\geq
  0}(\widetilde{\Pc} \cap \ZZ^{N+1})\}.$$
Thus for $n \geq \mu_\text{hole}(\Pc)$ and an element $\alpha \in \ell
(n\Pc) \cap \ZZ^N$, we can write $(\alpha,\ell n) \in \Cc(\Pc) \cap
\ZZ^{N+1}$ as
$
(\alpha_0,n_0) + (\alpha_1,1) + \ldots + (\alpha_r,1)
$
with $n_0 \le n$. These summands can now be grouped into $\ell$
elements of $n\widetilde{\Pc} \cap \ZZ^{N+1}$. \\
%
$\bullet$ \underline{$\mu_\text{hole}(\Pc) \leq \mu_\text{Ehr}(\Pc)$:}
This inequality follows from the fact that $\delta_i$ counts box
points of degree $i$ in a triangulation of $\Pc$, cf.\ \cite[Pf.\ of
Thm.\ 3.12]{BR}. Here is the argument.

Let $(\alpha,\mu_\text{hole}(\Pc)) \in \text{Box}(\Pc)$ attain
$\mu_\text{hole}(\Pc)$. Then, by definition of $\text{Box}(\Pc)$, we
can describe $(\alpha,\mu_\text{hole}(\Pc))$ as a linear combination
of $(d+1)$ linearly independent lattice vectors in $\widetilde{\Pc}$.
Say, $(\alpha,\mu_\text{hole}(\Pc)) = \sum_{i=0}^{d} r_i (v_i,1)$ with
$0 \le r_i <1$. Let $\beta=\sum_{i=0}^{d} (1-r_i) (v_i,1) \in
\ZZ^{N+1}$. Since $1-r_i > 0$ and $\sum_{i=0}^{d} (1-r_i) =
d+1-\mu_\text{hole}(\Pc)$, one has $\beta \in
(d+1-\mu_\text{hole}(\Pc))(\Pc \setminus \partial \Pc) \cap \ZZ^N$,
certifying $d+1-\mu_\text{Ehr}(\Pc) \leq d+1-\mu_\text{hole}(\Pc)$.
\\
$\bullet$ \underline{$\mu_\text{va}(\Pc) \leq \mu_\text{Hilb}(\Pc)$:}
Let $\alpha$ be a vertex of $\Pc$, and let $\beta$ be an integer point
in the cone generated by $\Pc-\alpha$. Then, for a sufficiently large integer $k$, 
we have $(\beta+k\alpha,k) \in \Cc(\Pc) \cap
\ZZ^{N+1}$. This point can be written as a nonnegative integral linear
combination of $\Hc(\Cc(\Pc))$: $(\beta+k\alpha,k) =
\sum_{h \in \Hc(\Cc(\Pc))} w_h h$ with $w_h \in \ZZ_{\geq 0}$. This
implies that
$$(\beta,0) \ = \ \sum_{h \in \Hc(\Cc(\Pc))} w_h \bigl[ (h+
  (\mu_\text{Hilb}(\Pc)-\deg(h)) (\alpha,1)) \ - \
  \mu_\text{Hilb}(\Pc) (\alpha,1) \bigr]$$
is a nonnegative integral linear combination of integer points in
$\mu_\text{Hilb}(\Pc) (\Pc-\alpha) \times \{0\}$, showing that
$\mu_\text{Hilb}(\Pc)\Pc$ is very ample.
\\
$\bullet$ \underline{$\mu_\text{Hilb}(\Pc) \leq
  \mu_\text{hole}(\Pc)$:} 
This follows from $\Hc(\Cc(\Pc)) \subseteq \text{Box}(\Pc) \cup (\widetilde{\Pc} \cap \ZZ^{N+1})$.
\\
$\bullet$ \underline{$\mu_\text{hole}(\Pc) \leq d-1$:}
Here is the argument from \cite{LiuTrotterZiegler}.
Let $\Sc \subset \RR^N$ be an empty simplex of dimension $d$ and let
$v_0,v_1,\ldots,v_d$ be its vertices.
Given $v \in \text{Box}(\Sc)$ there are $r_0,r_1,\ldots,r_d$ with
$0 \leq r_i < 1$ such that $v=\sum_{i=0}^d r_i(v_i,1)$. Since $r_i
<1$, one has $\deg(v)=\sum_{i=0}^d r_i < d+1$.
Suppose that $\deg(v)=d$. Then all $r_i$ must be positive. 
Moreover, we have $\sum_{i=0}^d(1-r_i)=1$ and $0 < 1-r_i <1$. 
Thus the integer point $\sum_{i=0}^d (1-r_i) v_i$ belongs to the
interior of $\Sc$, 
a contradiction. Hence $\deg(v) \leq d-1$.
\end{proof}

\section{Proper inequalities of \eqref{diagram}}

In this section, we present a lot of examples of integral convex polytopes. 
Each of the examples satisfies a proper inequality in \eqref{diagram}. 

Before giving them, we see that there exists an integral convex polytope $\Pc$ attaining 
$\mu_\text{va}(\Pc)=\mu_\text{midp}(\Pc)=\mu_\text{idp}(\Pc)=\mu_\text{Hilb}(\Pc)=\mu_\text{hole}(\Pc)=\mu_\text{Ehr}(\Pc)$. 

\begin{Theorem}\label{Berkeley}
Given integers $d \geq 3$ and $2 \leq j \leq d - 1$, 
there exists an empty simplex $\Pc$ of dimension $d$ with 
$\mu_\text{{\em va}}(\Pc)=\mu_\text{{\em midp}}(\Pc)=\mu_\text{{\em idp}}(\Pc)
=\mu_\text{{\em Hilb}}(\Pc)=\mu_\text{{\em hole}}(\Pc)=\mu_\text{{\em Ehr}}(\Pc)=j$. 
\end{Theorem}
\begin{proof}
Fix positive integers $d$ and $j$ with $d \geq 3$ and $2 \leq j \leq d-1$. 
Write $\Pc(d,j) \subset \RR^{d}$ for the convex hull of $\{ v_0,v_1,\ldots,v_d \} \subset \ZZ^N$, 
where $v_0={\bf 0}=(0,\ldots,0)$ and $v_i,i=1,\ldots,d,$ is the $i$th row vector of the matrix 
\begin{eqnarray}\label{mat}
\begin{pmatrix}
1      &0      &0      &\cdots &\cdots &\cdots &0 \\
0      &1      &\ddots &       &       &       &0 \\
\vdots &\ddots &\ddots &\ddots &       &       &\vdots \\
\vdots &       &\ddots &\ddots &\ddots &       &\vdots \\
\vdots &       &       &\ddots &\ddots &\ddots &\vdots \\
0      &       &       &       &\ddots &1      &0 \\
1      &\cdots &1      &0      &\cdots &0      &j
\end{pmatrix}. 
\end{eqnarray}
Here there are $j$ 1's in the $d$th row. 
Let $\Pc=\Pc(d,j)$ and $(\delta_0,\delta_1,\ldots,\delta_d)$ be the $\delta$-vector of $\Pc$.
We prove that $\Pc$ enjoys the required properties. 

Let $\eb_i \in \RR^d$ for $i=1,\ldots,d$ denote a unit coordinate vector of $\RR^d$. 
Since the determinant of \eqref{mat} is $j$, the normalized volume of $\Pc$ is $j$. Moreover,  
\begin{eqnarray}\label{hilb}
\frac{q}{j}(v_0,1)+\frac{q}{j}\sum_{i=1}^j(v_i,1)+\frac{j-q}{j}(v_d,1)=(\eb_1+\cdots+\eb_j+(j-q)\eb_d,q+1), 
\end{eqnarray}
where $q=1,\ldots,j-1$. 
Thus $\delta_q \geq 1$ for $q=2,\ldots,j$ (\cite[Proposition 27.7]{HibiRedBook}). 
Since $\sum_{i=0}^d \delta_i =j$ and each $\delta_i$ is nonnegative, one has 
$$(\delta_0,\delta_1,\ldots,\delta_d)=(1,0,\underbrace{1,1,\ldots,1}_{j-1},0,\ldots,0).$$ 
In particular, $\mu_\text{Ehr}(\Pc)=j$. 
Moreover, from $\delta_1=0$, $\Pc$ is an empty simplex. 
Thus, once we show $\mu_\text{va}(\Pc) \geq j$, we conclude that 
$\Pc$ has the desired properties by \eqref{diagram} in Theorem \ref{invariant}. 

Using \eqref{hilb}, one can show without difficulty that the Hilbert basis $\Hc(\Cc(\Pc))$ is 
\[(\widetilde{\Pc} \cap \ZZ^{d+1}) \cup \{(\eb_1+\cdots+\eb_j+(j-q+1)\eb_d,q) : q =2,\ldots,j\}.\]
Now, we show that $k \Pc$ cannot be very ample for $1 \leq k < j$. 
Let $k$ be an integer with $1 \leq k <j$ 
and $m$ the least common multiple of $k$ and $j$.
Write $m=kg$ with $g \geq 2$. Let 
\[
\alpha=(\alpha_0, m+\ell k), \; 
\alpha_0=(m-j+\ell k +1)\eb_1 + \eb_2+\cdots+\eb_j+\eb_d, 
\]
where $\ell$ is an arbitrary nonnegative integer. Since 
\[\alpha=(m-j+\ell k)(v_1,1)+(\eb_1+\cdots+\eb_j+\eb_d,j), \]
it follows that $\alpha$ belongs to $\Cc(\Pc) \cap \ZZ^{d+1}$. 
This implies, first, that $\alpha \not\in \ZZ_{\geq 0} (\widetilde{\Pc} \cap \ZZ^{d+1})$ 
and, second, that $$\alpha_0 \in (m+\ell k)\Pc \cap \ZZ^N = (g+\ell)(k \Pc) \cap \ZZ^N.$$ 
If $k\Pc$ was very ample, then for sufficiently large, we could write 
$\alpha_0=\alpha_1+\cdots+\alpha_{g+\ell}$, 
where $\alpha_1,\ldots,\alpha_{g+\ell} \in k \Pc \cap \ZZ^N$. 
Then the $d$th coordinate of each of $\alpha_1,\ldots,\alpha_{g+\ell}$ must be 0 or 1. 
Consider $(\alpha_i,k) \in \Cc(\Pc) \cap \ZZ^{d+1}$ for $1 \leq i \leq g+\ell$. 
Since each $h \in \Hc(\Cc(\Pc))$ with $2 \leq \deg h \leq k$ is of the form 
$$h = (\eb_1+\cdots+\eb_j+(j-i+1)\eb_d,i), $$
where $i=2,\ldots,k$, none of $\alpha_1,\ldots,\alpha_{g+\ell}$ can be expressed 
by using such elements. Thus each of $(\alpha_1,k),\ldots,(\alpha_{g+\ell},k)$ must 
be written as the sum of $k$ elements belonging to $\widetilde{\Pc} \cap \ZZ^{d+1}$. 
It then follows that $\alpha=(\alpha_0,m+\ell k)$ can be written as the sum of $(m+\ell k)$ elements 
belonging to $\widetilde{\Pc} \cap \ZZ^{d+1}$. 
This contradicts $\alpha \not\in \ZZ_{\geq 0}(\widetilde{\Pc} \cap \ZZ^{d+1})$. 
Consequently, $k\Pc$ cannot be very ample, as required. 
\end{proof}

The following three examples (Example \ref{Higashitani}, \ref{rei} and \ref{Ehr}) 
show the existence of integral convex polytopes 
attaining each of proper inequalities of \eqref{diagram}. 

\begin{Example}[$\mu_\text{va}(\Pc) < \mu_\text{midp}(\Pc)$ and
  $\mu_\text{va}(\Pc) < \mu_\text{Hilb}(\Pc)$]\label{Higashitani}
In \cite[Theorem 0.1]{H}, for each $d \geq 3$, the fourth author establishes an example of 
a very ample integral convex polytope not having (IDP). 
From the proof there, we know that this polytope satisfies 
$\mu_\text{va}(\Pc)=1$ but $\mu_\text{midp}(\Pc) = \mu_\text{Hilb}(\Pc)= 2$. 
\end{Example}

\begin{Example}[$\mu_\text{midp}(\Pc) < \mu_\text{idp}(\Pc) <
  \mu_\text{hole}(\Pc)$ and $\mu_\text{Hilb}(\Pc) <
  \mu_\text{hole}(\Pc)$]
\label{rei}
Let $d=2m-1$ with $m \geq 4$ and $\Pc$ be the integral simplex whose
vertex set is
$\{{\bf 0}, \eb_1, \ldots,\eb_{d-1},
(m-1)\eb_1+\cdots+(m-1)\eb_{d-1}+m \eb_d\}$.
Then it is immediate to see that one has  
$$\text{Box}(\Pc)=\{(j,\ldots,j,2j) : j=1,\ldots,m-1\}.$$ 
Thus, for $j=2,\ldots,m-1$, we can write
$(j,\ldots,j,2j)=j(1,\ldots,1,2)$.
This implies that $\{x \in \Hc(\Cc(\Pc)) : \text{deg}(x) \geq 2\}=\{(1,\ldots,1,2)\}$.
Hence $\mu_\text{Hilb}(\Pc)=2$, while $\mu_\text{hole}(\Pc)=2m-2=d-1$.


Moreover, 
we also know that
\begin{equation}\label{tousiki}
\begin{aligned}
&\Cc(\Pc) \cap \ZZ^{d+1} = \ZZ_{\geq 0}(\widetilde{\Pc} \cap
\ZZ^{d+1}) \cup \\
&\quad\quad\quad\quad\quad\quad\quad\quad\quad
\{(j,\ldots,j,2j)+ x : 1 \leq j \leq m-1, x \in \ZZ_{\geq
  0}(\widetilde{\Pc} \cap \ZZ^{d+1})\}. 
\end{aligned} 
\end{equation}
It then follows from \eqref{tousiki} that for every element $\alpha$
in $2k \Pc \cap \ZZ^d$
with $k \geq 1$, we can write $\alpha=\alpha_1+\cdots+\alpha_\ell$, 
where $k \leq \ell \leq 2k$ and $\alpha_i \in \Pc \cap \ZZ^d$ or
$\alpha_i=(1,\ldots,1) \in 2 \Pc \cap \ZZ^d$.
By rewriting appropriately, we can express $\alpha=\alpha_1'+\cdots+\alpha_k'$, 
where $\alpha_i' \in 2 \Pc \cap \ZZ^d$. This means that $\mu_\text{midp}(\Pc)=2$. 

On the other hand, we have $(3,\ldots,3) \in 2(3 \Pc) \cap \ZZ^d$ but 
$(3,\ldots,3) \not\in \{\alpha + \beta : \alpha, \beta \in 3\Pc \cap \ZZ^d\}$ 
because of $3 \leq m-1$ and \eqref{tousiki}. Similarly, $k \Pc$ does not possess 
(IDP) when $k$ is odd and $k \leq m-1$. 
However, if $k \geq m$, then $\Qc=k \Pc$ has (IDP). 
In fact, for $\alpha \in \ell \Qc \cap \ZZ^d$ with $\ell \geq 2$, 
since $\ell k \geq 2k \geq 2m$ and $\text{Box}(\Pc)$ has at most 
degree $2m-2$ elements, we can express $\alpha$ as 
$\alpha=(j,\ldots,j) + \alpha'$, 
where $1 \leq j \leq m-1$ and $\alpha' \in \{\alpha_1'+\cdots + \alpha_q' : 
\alpha_i \in \Pc \cap \ZZ^d, q \geq 2\}$. 
Thanks to $q \geq 2$, $\alpha$ can be described as a sum of $\ell$
elements belonging to $\Qc \cap \ZZ^d$. Hence we obtain 
\begin{eqnarray*}
\mu_\text{idp}(\Pc)=
\begin{cases}
m-1 \;&\text{ if $m$ is odd}, \\
m  &\text{ if $m$ is even}. 
\end{cases}
\end{eqnarray*}

Therefore, in summary, 
$$2=\mu_\text{Hilb}(\Pc)=\mu_\text{midp}(\Pc)<
\mu_\text{idp}(\Pc)=2 \left\lfloor \frac{m}{2} \right\rfloor < \mu_\text{hole}(\Pc)=2m-2.$$ 
\end{Example}

\begin{Example}[$\mu_\text{hole}(\Pc) <
  \mu_\text{Ehr}(\Pc)$]\label{Ehr}
When $\Pc$ is an integral convex polytope of dimesnion $d$ 
which does not have (IDP) and contains an integer point in its interior, 
one has $\mu_\text{Ehr}(\Pc)=d$ but $\mu_\text{hole}(\Pc) \leq d-1$. 
For example, let us consider the integral simplex $\Pc$ of dimension $d \geq 3$ 
whose vertices are {\bf 0} and the row vectors of $d \times d$ matrix 
\begin{eqnarray*}
\begin{pmatrix}
1      &0      &\cdots &\cdots &0 \\
0      &1      &\ddots &       &\vdots \\
\vdots &\ddots &\ddots &\ddots &\vdots \\
0      &\cdots &0      &1      &0 \\
d+1    &\cdots &\cdots &d+1    &d+2
\end{pmatrix}. 
\end{eqnarray*}
Let $v_0={\bf 0}$ and let $v_i$ denote the $i$th row vector. Then the integer point 
$$\frac{2}{d+2}v_0+\frac{1}{d+2}(v_1+\cdots+v_d)=(1,\ldots,1)$$ 
is contained in the interior of $\Pc$, implying $\mu_\text{Ehr}(\Pc)=d$. On the other hand, 
it is easy to see that 
$$\text{Box}(\Pc)=\left\{\left(\left\lfloor\frac{d+1}{2}\right\rfloor+1,\ldots,\left\lfloor\frac{d+1}{2}\right\rfloor+1, 
\left\lfloor\frac{d+1}{2}\right\rfloor \right) \right\},$$ 
implying $\mu_\text{hole}(\Pc)=\left\lfloor\frac{d+1}{2}\right\rfloor$. 
\end{Example}

Next, we consider possible relations between $\mu_\text{Hilb}(\Pc)$ and $\mu_\text{midp}(\Pc)$ 
and also between $\mu_\text{Hilb}(\Pc)$ and $\mu_\text{idp}(\Pc)$. 
As is shown below, there are no relations between them. 
\begin{Example}[$\mu_\text{Hilb}(\Pc)<\mu_\text{midp}(\Pc)$]\label{rei3}
The following integral simplex $\Pc$ of dimension 13 has
$\mu_\text{Hilb}(\Pc)=3$ but $\mu_\text{midp}(\Pc) = 4$:
Let $\Pc$ be a convex hull of {\bf 0} and the row vectors of the matrix 
$\begin{pmatrix}
A &0 \\
0 &B
\end{pmatrix}$, where $A$ (resp. $B$) is a $7 \times 7$ (resp. $6 \times 6$) matrix such that 
\begin{eqnarray*}
\begin{pmatrix}
1      &0      &\cdots &\cdots &0 \\
0      &1      &\ddots &       &\vdots \\
\vdots &\ddots &\ddots &\ddots &\vdots \\
0      &\cdots &0      &1      &0 \\
3      &\cdots &\cdots &3      &4
\end{pmatrix} \text{ (resp. }
\begin{pmatrix}
1      &0      &\cdots &\cdots &0 \\
0      &1      &\ddots &       &\vdots \\
\vdots &\ddots &\ddots &\ddots &\vdots \\
0      &\cdots &0      &1      &0 \\
1      &\cdots &\cdots &1      &2
\end{pmatrix}). 
\end{eqnarray*}
Notice that $A$ corresponds to the polytope in Example \ref{rei} in the case of $m=4$. 
It can be verified that 
$$\Hc(\Cc(\Pc))=(\widetilde{\Pc} \cap \ZZ^{d+1}) \cup 
\{(\underbrace{1,\ldots,1}_7,\underbrace{0,\ldots,0}_6,2), 
(\underbrace{0,\ldots,0}_7,\underbrace{1,\ldots,1}_6,3)\}.$$ 
Thus $\mu_\text{Hilb}(\Pc)=3$. On the other hand, neither $2\Pc$ nor $3\Pc$ possesses (IDP). 
(In fact, $(0,\ldots,0,1,\ldots,1,4) \in \Cc(2\Pc) \cap \ZZ^{d+1} \setminus \ZZ_{\geq 0}(\widetilde{2\Pc} \cap \ZZ^{d+1})$ 
and $(3,\ldots,3,0,\ldots,0,6) \in \Cc(3\Pc) \cap \ZZ^{d+1} \setminus \ZZ_{\geq 0}(\widetilde{3\Pc} \cap \ZZ^{d+1})$.) 
Hence $\mu_\text{midp}(\Pc) \geq 4$. In fact, one can show that $\mu_\text{midp}(\Pc)=4$. 
\end{Example}

Note that an example attaining $\mu_\text{Hilb}(\Pc)<\mu_\text{idp}(\Pc)$ 
has been already given in Example \ref{rei}. 
The following theorem gives an example attaining 
both $\mu_\text{Hilb}(\Pc)>\mu_\text{midp}(\Pc)$ and $\mu_\text{Hilb}(\Pc)>\mu_\text{idp}(\Pc)$. 

\begin{Theorem}\label{Boston}
Given an integer $d \geq 4$, 
there exists an integral convex polytope $\Pc$ of dimension $d$ 
such that $\mu_\text{{\em Hilb}}(\Pc) = d - 1$ and $\mu_\text{{\em midp}}(\Pc)=\mu_\text{{\em idp}}(\Pc)=d-2$. 
\end{Theorem}
\begin{proof}
Work with a fixed integer $d \geq 4$ and let $M=d(d-2)+1$. 
We define $v_{j} \in \ZZ^{d}$, $1 \leq i \leq d$, as follows:
\begin{eqnarray*}
v_j=
\begin{cases}
{\bf 0}, &j=0, \\
\eb_j, &j=1,\ldots,d-1, \\
\eb_1+\cdots+\eb_{d-1}+M\eb_d, &j=d.
\end{cases}
\end{eqnarray*}
Let $v_j'=v_j+\eb_d$ for $j=0,1,\ldots,d$. 
We write $\Pc \subset \RR^d$ for the 
integral convex polytope of dimension $d$ 
with the vertices $v_0, v_{1}, \ldots, v_{d}$ and $v'_0, v'_{1}, \ldots, v'_{d}$.
Such a convex polytope appears in \cite[Section 1]{Ogata}. 

It will be proved that $\Pc$ enjoys the required properties. 


\noindent
{\bf (First Step)}
First of all, the Hilbert basis $\Hc(\Cc(\Pc))$ must be computed. 
If $q \in \{1,\ldots,M-1\}$, then there exist unique integers 
$k$ and $s$ with $1 \leq k \leq d-2$ and $1 \leq s \leq d$ such that $q=(k-1)d+s$.
Since 
\begin{align*}
&\frac{(d-2)s-k+1}{M}(v_0,1)+\frac{M-q}{M}\sum_{j=1}^{d-1}(v_j,1)+\frac{q}{M}(v_d,1) \\
&\quad\quad\quad\quad\quad\quad\quad\quad\quad\quad\quad\quad\quad\quad
=(\eb_1+\cdots+\eb_{d-1}+q\eb_d, d-k) \in \ZZ^{d+1}, 
\end{align*}
it follows that $(\eb_1+\cdots+\eb_{d-1}+q\eb_d, d-k) \in \Cc(\Pc) \cap \ZZ^{d+1}$. 
When $k=1$ and $s=1,\ldots,d-1$, one has $q=s$ and 
\[
(\eb_1+\cdots+\eb_{d-1}+q\eb_d, d-1)=\sum_{j=1}^s (v_j',1) + \sum_{j=s+1}^{d-1} (v_j,1). 
\]
Hence $(\eb_1+\cdots+\eb_{d-1}+q\eb_d, d-1)$ 
cannot belong to $\Hc(\Cc(\Pc))$.
Now, it is routine work to show that, by considering the facets of the cone $\Cc(\Pc)$, 
the Hilbert basis $\Hc(\Cc(\Pc))$ coincides with 
\[
(\widetilde{\Pc} \cap \ZZ^{d+1})
\cup \{(\eb_1+\cdots+\eb_{d-1}+q\eb_d,d-\lfloor (q-1)/d \rfloor -1) : q=d,\ldots,M-1\}.
\]
Thus, in particular, $\mu_\text{Hilb}(\Pc)=d-1$. 


\noindent
{\bf (Second Step)}
Let 
\[u_s^{(k)}=\eb_1+\cdots+\eb_{d-1}+((k-1)d+s)\eb_d, \]
where $k=2,\ldots,d-2$ and $s=1,\ldots,d$, and 
\[
u=\eb_1+\cdots+\eb_{d-1}+d \eb_d.
\]
One can easily see the identities 
\begin{align*}
&(u,d-1)+(v_i,1)=(v_i',1)+\sum_{j=1}^{d-1}(v_j', 1) \;\;\text{for}\;\;i=0,1,\ldots,d, \\
&(u,d-1)+(v_i',1)=(v_0,1)+(v_i,1)+(u_1^{(2)},d-2) \;\;\text{for}\;\;i=0,1,\ldots,d, \\
&(u,d-1)+(u_s^{(d-2)},2)=(v_0,1)+(v_d,1)+\sum_{j=1}^{s-1}(v_j',1)+\sum_{j=s}^{d-1}(v_j,1), \\
&(u,d-1)+(u_s^{(k)},d-k)=(u_s^{(k+1)},d-k-1)+\sum_{j=0}^{d-1}(v_j,1) \;\;\text{for}\;\;
k=2,\ldots,d-3, \\
&(u,d-1)+(u,d-1)=(u_d^{(2)},d-2)+\sum_{j=0}^{d-1}(v_j,1). 
\end{align*}
It then follows that
\[(\Cc(\Pc) \cap \ZZ^{d+1}) \setminus \{(u,d-1)\} = \ZZ_{\geq 0} (\Hc(\Cc(\Pc)) \setminus \{(u,d-1)\}).\]
Moreover, if $k+k' \geq d$, then
\begin{align*}
&(u_s^{(k)},d-k)+(u_{s'}^{(k')},d-k')=\\
&\;\;\;\;\;\;\;\;\;\;\;\;\;\;\;\;\;
\begin{cases}
(v_d,1)+(u_{s+s'-1}^{(k+k'-d+1)}, 2d-k-k'-1),
&\;\;\text{if}\;\;s+s'\leq d+1, \\
(v_0,1)+(v_d,1)+(u_{s+s'-1-d}^{(k+k'-d+2)}, 2d-k-k'-2),
&\;\;\text{if}\;\;s+s'\geq d+2. 
\end{cases}\end{align*}
If $k+k' \leq d-1$, then 
\begin{align*}
&(u_s^{(k)},d-k)+(u_{s'}^{(k')},d-k')=\\
&\;\;\;\;\;\;\;\;
\begin{cases}
(u_s^{(k+k'-1)},d-k-k'+1)+\sum_{j=1}^{s'}(v_j',1)+\sum_{j=s'+1}^{d-1}(v_j,1), 
&\;\;\text{if}\;\;s'\leq d-1, \\
(u_s^{(k+k')},d-k-k')+\sum_{j=0}^{d-1}(v_j,1), 
&\;\;\text{if}\;\;s'= d. 
\end{cases}\end{align*}
Consequently, when we write $\alpha \in \Cc(\Pc) \cap \ZZ^{d+1}$ 
by using at least two $u_s^{(k)}$'s, we can reduce one $u_s^{(k)}$.
Hence $\alpha$ can be expressed by using at most one $u_s^{(k)}$ belonging to $\Hc(\Cc(\Pc))$. 


\noindent
{\bf (Third Step)}
Let $\Pc'=(d-2)\Pc$ and $\alpha \in n \Pc' \cap \ZZ^d$.
Since $d \geq 4$, one has $n(d-2) \neq d-1$. Thus $\alpha \not=u$. 
By virture of the (Second Step), there exists an expression of $\alpha$ of the form 
\[
(\alpha, n(d-2)) = v' + \sum_{j=1}^{n(d-2)-\deg(v')}(v_j'',1), 
\]
where $v' \in \{(u_s^{(k)},d-k) : k=2,\ldots,d-2, s=1,\ldots,d\}$ and each $v_j'' \in \Pc \cap \ZZ^d$. 
Since the degree of $v'$ is at most $d-2$, there exists an expression of $\alpha$ of the form
\[
\alpha=\alpha_1+\cdots+\alpha_n
\] 
with each $\alpha_i \in \Pc' \cap \ZZ^d$. Hence $\Pc'$ possesses (IDP). 


\noindent
{\bf (Fourth Step)}
By $\mu_\text{idp}(\Pc) \leq d-1$, $k\Pc$ possesses (IDP) for $k \geq d-1$. 
Thus, by (Third Step), we obtain $\mu_\text{idp}(\Pc) \leq d-2$. 
Now it is easy to see that for $r < d-2$, $r \Pc$ never possesses (IDP). 
Therefore, we conclude that $\mu_\text{midp}(\Pc)=\mu_\text{idp}(\Pc)=d-2$, as required. 
\end{proof}


\section{Restrictions on invariants}

In this section, we discuss more restrictions on the invariants. 
We consider the following question. 

\begin{Question}\label{toi}{\em 
Let $d, a_1, a_2, a_3, a_4, a_5, a_6$ be positive integers satisfying 
\begin{align*}
a_1 \leq a_2 \leq a_3 \leq a_5 \leq a_6 \leq d \;\text{ and }\;  a_1 \leq a_4 \leq a_5 \leq d-1. 
\end{align*}
Then does there exist an integral convex polytope $\Pc$ of dimension $d$ such that 
\begin{eqnarray*}
&\mu_\text{va}(\Pc)=a_1, \mu_\text{midp}(\Pc)=a_2, \mu_\text{idp}(\Pc)=a_3, 
\mu_\text{Hilb}(\Pc)=a_4, \mu_\text{hole}(\Pc)=a_5 \\
&\text{ and } \mu_\text{Ehr}(\Pc)=a_6 \;\text{?}
\end{eqnarray*}
}\end{Question}

From some easy observations, we cannot assign these positive integers freely. 
In fact, it is obvious that 
\begin{itemize}
\item if either $\mu_\text{midp}(\Pc)$ or $\mu_\text{Hilb}(\Pc)$ is 1, 
then $\mu_\text{va}(\Pc)=\mu_\text{midp}(\Pc)=\mu_\text{idp}(\Pc)=\mu_\text{Hilb}(\Pc)
=\mu_\text{hole}(\Pc)=1$; 
\item if $\mu_\text{midp}(\Pc)<\mu_\text{idp}(\Pc)$, 
then $\mu_\text{idp}(\Pc) \geq \mu_\text{midp}(\Pc) +2$. 
\end{itemize}


Moreover, we also see non-trivial restrictions. 
\begin{Theorem}\label{teiri} The following assertions hold: 
\begin{itemize}
\item[(1)] if $\mu_\text{{\em midp}}(\Pc) \geq (d-1)/2$, 
then $\mu_\text{{\em idp}}(\Pc)=\mu_\text{{\em midp}}(\Pc)$; 
\item[(2)] if $\mu_\text{{\em midp}}(\Pc) \leq (d-1)/2$, 
then $\mu_\text{{\em idp}}(\Pc) \leq (d-3)/2+\mu_\text{{\em midp}}(\Pc)$. 
\end{itemize}
\end{Theorem}

Before proving these, we prove the following lemma. 
\begin{Lemma}[cf. {\cite[Theorem 2.2.12]{CLS}}]\label{hodai}
Let $\Pc \subset \RR^N$ be an integral convex polytope of dimension $d$. 
Given $\alpha \in n \Pc \cap \ZZ^N$ for $n \geq d-1$, $\alpha$ can be written like 
$$\alpha=\alpha'+\alpha_1+\cdots+\alpha_{n-d+1},$$ 
where $\alpha' \in (d-1) \Pc \cap \ZZ^N$ and $\alpha_1,\ldots,\alpha_{n-d+1} \in \Pc \cap \ZZ^N$. 
\end{Lemma}
A proof of this lemma appears in the proof of \cite[Theorem 2.2.12]{CLS}.

\begin{proof}[Proof of Theorem \ref{teiri}]
(1) It suffices to show that for an integral convex polytope $\Pc \subset \RR^N$ of dimension $d$ 
with $\mu_\text{midp}(\Pc) \geq (d-1)/2$, $n\Pc$ possesses (IDP) for every $n \geq \mu_\text{midp}(\Pc)$. 

Let $k=\mu_\text{midp}(\Pc)$. 
Let $n \geq k$ and let $\alpha \in m(n\Pc) \cap \ZZ^N$ for $m \geq 2$. 
Since $mn \geq 2n \geq 2k \geq d-1$, 
thanks to Lemma \ref{hodai}, we obtain 
$$\alpha=\alpha'+\alpha_1+\cdots+\alpha_{mn-d+1},$$ 
where $\alpha' \in (d-1)\Pc \cap \ZZ^N$ and 
$\alpha_1,\ldots,\alpha_{mn-d+1} \in \Pc \cap \ZZ^N$. 
Moreove, since $k\Pc$ has (IDP), 
there are $\alpha_1'$ and $\alpha_2'$ in $k\Pc \cap \ZZ^N$ such that 
$$\alpha_1'+\alpha_2'=\alpha'+\alpha_{mn-2k+1}+\cdots+\alpha_{mn-d+1} \in 2k \Pc \cap \ZZ^N.$$ 
Therefore, 
$$\alpha=\underbrace{\alpha_1'+\sum_{i=1}^{n-k}\alpha_i}_{n \Pc \cap \ZZ^N}+
\underbrace{\alpha_2'+\sum_{i=n-k+1}^{2(n-k)}\alpha_i}_{n \Pc \cap \ZZ^N}+
\sum_{i=1}^{m-2}\underbrace{\sum_{j=1}^n\alpha_{2(n-k)+(i-1)n+j}}_{n \Pc \cap \ZZ^N}.$$
This implies that $n\Pc$ possesses (IDP). \\
(2) It suffices to prove that for an integral convex polytope $\Pc \subset \RR^N$ 
of dimension $d$ with $\mu_\text{midp}(\Pc) \leq (d-1)/2$, 
$n\Pc$ possesses (IDP) for every $n \geq (d - 3)/2+\mu_\text{midp}(\Pc)$. 

For $m \geq 2$, let $\alpha \in m(n\Pc) \cap \ZZ^N$. 
Since $mn \geq 2n \geq d-3 + 2 \mu_\text{midp}(\Pc) \geq d-1$, thanks to Lemma \ref{hodai}, we obtain 
$$\alpha=\alpha'+\alpha_1+\cdots+\alpha_{mn-d+1},$$ 
where $\alpha' \in (d-1)\Pc \cap \ZZ^N$ and 
$\alpha_1,\ldots,\alpha_{mn-d+1} \in \Pc \cap \ZZ^N$. 

Let $k=\mu_\text{midp}(\Pc)$ and $\ell = \min\{ i : ik \geq d-1\}$. 
Since $k\Pc$ has (IDP), an element 
$\alpha'+\alpha_{mn-\ell k +1}+\cdots+\alpha_{mn - d+1}$ belonging to $\ell (k \Pc) \cap \ZZ^N$ 
can be written such as $\alpha_1'+\cdots+\alpha_\ell'$, 
where $\alpha_i' \in k \Pc \cap \ZZ^N$. 
Remark that $\ell k - d+1 \leq mn- d +1$ because 
$$mn - \ell k \geq 2n - \ell k \geq d-3+2k-\ell k \geq d-2 - (\ell - 1) k \geq 0.$$
Thus $\alpha$ can be rewritten as 
$$\alpha=\alpha_1'+\cdots+\alpha_\ell' + \alpha_1+\cdots+\alpha_{mn-\ell k}.$$ 
Let $p=\lfloor n/k \rfloor$ and $q=n-pk$, i.e., $n=pk+q$ with $0 \leq q \leq k-1$. 
When $p \geq \ell$, since $n \geq \ell k$, 
it follows easily that $\alpha$ can be written as a sum of $m$ elements of $n \Pc \cap \ZZ^N$. 
Assume that $p < \ell$. Then we have $mn-\ell k \geq q$ and $n - (\ell-p)k \geq 0$. In fact, 
\begin{align*}
mn - \ell k -q &\geq 2n-\ell k -q=n-(\ell-p)k \\
&\geq d-3+2k-\ell k -(k-1) = d-2 - (\ell -1 )k \geq 0.
\end{align*}
Thus we obtain that 
$$\alpha=\underbrace{\sum_{i=1}^p\alpha_i'+\sum_{j=1}^q \alpha_j}_{n \Pc \cap \ZZ^N}
+\underbrace{\sum_{i=p+1}^\ell \alpha_i'+\sum_{j=1}^{n-(\ell-p)k}\alpha_{q+j}}_{n \Pc \cap \ZZ^N}
+\sum_{r=1}^{m-2}\underbrace{\sum_{i=1}^n \alpha_{n-(\ell-p)k+q+(r-1)n+i}}_{n \Pc \cap \ZZ^N}.$$
This says that $n \Pc$ possesses (IDP), as desired. 
\end{proof}

As an immediate corollary of Theorem \ref{teiri}, we obtain the following. 
\begin{Corollary}
If $\mu_\text{{\em idp}}(\Pc)=d-1$, then $\mu_\text{{\em midp}}(\Pc)=d-1$. 
\end{Corollary}
\begin{proof}
Suppose that $\mu_\text{midp}(\Pc)\not=d-1$. In particular, $\mu_\text{midp}(\Pc) < d-1$. 
If $\mu_\text{midp}(\Pc) \geq (d-1)/2$, 
then $\mu_\text{idp}(\Pc)=\mu_\text{midp}(\Pc)<d-1$ by Theorem \ref{teiri} (1). 
Moreover, if $\mu_\text{midp}(\Pc) \leq (d-1)/2$, then 
$\mu_\text{idp}(\Pc) \leq (d-3)/2 + \mu_\text{midp}(\Pc) \leq (d-3)/2+(d-1)/2 < d-1$ by Theorem \ref{teiri} (2). 
Hence, $\mu_\text{idp}(\Pc)$ is never equal to $d-1$, as desired. 
\end{proof}


\begin{Question}\label{mondai} Work with the same notation as above. 
\begin{itemize}
\item[(1)] Is there some relation between $\mu_\text{{\em midp}}(\Pc)$, $\mu_\text{{\em idp}}(\Pc)$ 
and $\mu_\text{{\em Hilb}}(\Pc)$?
For example, does there exist an example of polytope $\Pc$ such that 
$\mu_\text{{\em midp}}(\Pc) < \mu_\text{{\em Hilb}}(\Pc) < \mu_\text{{\em idp}}(\Pc)$?
\item[(2)] Is it true that 
$\mu_\text{{\em midp}}(\Pc)=\mu_\text{{\em idp}}(\Pc)=\mu_\text{{\em Hilb}}(\Pc)=2$ 
if $\mu_\text{{\em va}}(\Pc)=1$? 
\end{itemize}
\end{Question}


Moreover, the following is also interesting. 
\begin{Question}\label{minkowski}
For $n,m \geq 1$, does $(n+m)\Pc$ possess {\em (IDP)} if $n\Pc$ and $m\Pc$ possess {\em (IDP)}? 
\end{Question}

In general, Question \ref{minkowski} is no longer true for a Minkowski sum 
of polytopes having (IDP). For example, 
$\con(\{(0,0,0),(1,0,0),(0,1,0)\})$ and $\con(\{(0,0,0),(1,1,3)\})$ possess (IDP), 
but their Minkowski sum does not.

\begin{Remark}
In the book \cite{BG}, two notions closely related to (IDP) are described. 
Let $\Pc \subset \RR^N$ be an integral convex polytope of dimension $d$. 
We say that $\Pc$ is {\em integrally closed} if $\Pc$ satisfies 
$$\ZZ_{\geq 0}(\widetilde{\Pc}\cap \ZZ^{N+1})=\Cc(\Pc) \cap \ZZ^{N+1}$$ 
and we say that $\Pc$ is {\em normal} if $\Pc$ satisfies 
$$\ZZ_{\geq 0}(\widetilde{\Pc}\cap \ZZ^{N+1})=\Cc(\Pc) \cap \ZZ(\widetilde{\Pc} \cap \ZZ^{N+1}).$$ 
Then $\Pc$ satisfying (IDP) is equivalent $\Pc$ being integrally closed, 
but not necessarily equivalent to $\Pc$ being normal. 
\end{Remark}

\section{The case of dilated edge polytopes}

Finally, we discuss the case of edge polytopes. 

Recall that for a connected simple graph $G$ on the vertex set $\{1,\ldots,d\}$ 
with the edge set $E(G)$, the {\em edge polytope} of $G$ is 
the convex polytope $\Pc_G \subset \RR^d$ 
which is the convex hull of $\{\eb_i+\eb_j : \{i,j\} \in E(G)\}$. Also:
\begin{itemize}
\item An {\em odd} cycle is a cycle with odd length. 
\item A cycle $C$ in $G$ is called {\em minimal} if $C$ possesses no chord. 
\item A pair of distinct odd cycles $C$ and $C'$ in $G$ is said to be {\em exceptional} 
if there is no bridge between $C$ and $C'$ in $G$. 
\item We say that $G$ satisfies the {\em odd cycle condition} 
if each pair of distinct odd cycles is not exceptional. 
\end{itemize}


It is known that $\Pc_G$ has (IDP) if and only if $G$ satisfies 
the odd cycle condition (\cite[Corollary 2.3]{OH}). 

\begin{Proposition}
Let $G$ be a connected simple graph on 
$\{1,\ldots,d\}$ which does not satisfy the odd cycle condition. 
For distinct odd cycles $C_1$ and $C_2$, let 
\begin{align*}
m(C_1,C_2)=\frac{\ell(C_1)+\ell(C_2)}{2}, \;\; 
\end{align*} 
where $\ell(C_i)$ denotes the length of a cycle $C_i$. 
For an edge polytope $\Pc_G$, one has 
\begin{align*}
\mu_\text{{\em va}}(\Pc_G)&=\mu_\text{{\em Hilb}}(\Pc_G)\\
&=\max\left\{m(C,C') : (C,C') \; \text{{\em 
is an exceptional pair of minimal odd cycles}}\right\} 
\end{align*}
and $$\mu_\text{{\em hole}}(\Pc_G)=\max\left\{\sum_{i=1}^l m(C_{2i-1},C_{2i})\right\},$$ 
where $C_1,\ldots,C_{2l}$ are distinct and 
each of $(C_{2i-1},C_{2i})$ is an exceptional pair of minimal odd cycles. 
\end{Proposition}
%
%
%
%
\begin{proof}
In the case of edge polytopes, by \cite[Theorem 2.2]{OH}, 
$\Hc(\Cc(\Pc_G))$ and Box$(\Pc_G)$ can be written 
in terms of exceptional pairs of minimal odd cycles as follows: 
For a pair of minimal odd cycles $C$ and $C'$, let 
$$e(C,C')=\sum_{i \in V(C) \cup V(C')}\eb_i,$$ 
where $V(C)$ denotes the set of vertices of a cycle $C$. Then we have 
$$\Hc(\Cc(\Pc_G))=\{e(C, C') : (C, C') \text{ is an exceptional pair of minimal odd cycles}\}$$ 
and \begin{align}\label{box}
\text{Box}(\Pc_G)=\left\{\sum_{i=1}^le(C_{2i-1},C_{2i})\right\},\end{align}
where $C_1,\ldots,C_{2l}$ are distinct and each of $(C_{2i-1},C_{2i})$ is an exceptional pair of minimal odd cycles. 
Since $e(C_1,C_2) \in m(C_1,C_2)\Pc_G \cap \ZZ^N$, we obtain $\mu_\text{Hilb}(\Pc_G)=M$, 
where $M=\max\{m(C,C') : (C,C') \; \text{is an exceptional pair of minimal odd cycles}\}$, 
and $\mu_\text{hole}(\Pc_G) = \max\left\{\sum_{i=1}^l m(C_{2i-1},C_{2i})\right\}$. 
Our goal is to show $\mu_\text{va}(\Pc_G) \geq M$. 

Let $C_1$ and $C_2$ be distinct minimal odd cycles and 
let $(C_1,C_2)$ be exceptional and $M=m(C_1,C_2)$. 
Assume that $\{i_1,i_2\}$ is one edge in $C_1$. For each positive integer $\ell$, 
since there is no bridge between $C_1$ and $C_2$ and these cycles are minimal, one has 
$$e(C_1,C_2)+\ell(\eb_{i_1}+\eb_{i_2}) \in 
((M+\ell)\Pc_G \cap \ZZ^N) \setminus \{\alpha_1+\cdots+\alpha_{M+\ell} : \alpha_i \in \Pc_G \cap \ZZ^N\}.$$
Fix a positive integer $n$ with $n <M$. 
For every integer $m \geq \min\{ k : kn \geq M \}$, one has 
$$e(C_1,C_2)+(mn-M)(\eb_{i_1}+\eb_{i_2}) \in m(n \Pc_G) \cap \ZZ^N.$$ 
Since $n<M$, this integer point cannot be written as a sum of $m$ elements belonging to $n \Pc_G \cap \ZZ^N$. 
This says that $n \Pc_G$ is never very ample. Therefore, $\mu_\text{va}(\Pc_G) \geq M$, as desired. 
\end{proof}

On $\mu_\text{midp}(\Pc_G)$ and $\mu_\text{idp}(\Pc_G)$ of edge polytopes $\Pc_G$, 
these are not necessarily equal to $M$, although we still have 
$\mu_\text{idp}(\Pc_G) \geq \mu_\text{midp}(\Pc_G) \geq M$ because of $\mu_\text{va}(\Pc_G)=M$. 
\begin{Example}
Let us consider the graph $G$ on the vertex set $\{1,\ldots,25\}$ with the edge set 
\begin{eqnarray*}
E(G)=\{\{3i+1,3i+2\},\{3i+2,3i+3\},\{3i+1,3i+3\},\{3i+1,25\} : i=0,\ldots,7\}. 
\end{eqnarray*}
Then each of exceptional pairs of minimal odd cycles in this graph 
consists of two cycles of length 3. Thus 
we have $\mu_\text{va}(\Pc_G)=\mu_\text{Hilb}(\Pc_G)=3$. 
Moreover, since this graph contains four distinct exceptional pairs of minimal odd cycles, 
one has $\mu_\text{hole}(\Pc_G)=12$. In addition, we also see that $3\Pc_G$ has (IDP). 
Hence $\mu_\text{midp}(\Pc_G)=3$. 
On the other hand, neither $4\Pc_G$ nor $5\Pc_G$ has (IDP). In fact, 
\begin{align*}
&(\underbrace{1,\ldots,1}_{24},0) \in 3(4\Pc_G) \cap \ZZ^{25} \setminus 
\{\alpha_1+\alpha_2+\alpha_3 : \alpha_i \in 4\Pc_G \cap \ZZ^{25}\}\;\;\;\text{and} \\
&(\underbrace{1,\ldots,1}_{20},0,0,0,0,0) \in 2(5\Pc_G) \cap \ZZ^{25} \setminus 
\{\alpha_1+\alpha_2 : \alpha_i \in 5\Pc_G \cap \ZZ^{25}\}. 
\end{align*}
Thus $\mu_\text{idp}(\Pc_G) \geq 6$. In fact, one can show that $\mu_\text{idp}(\Pc_G)=6$. 

Let us consider the graph $G'$ on the vertex set $\{1,\ldots,30\}$ with the edge set 
$$E(G')=E(G) \cup \{\{25+i,26+i\},\{26,30\} : i=0,1,2,3,4\}.$$ 
Then there is an exceptional pair consisting of minimal odd cycles of length 3 and 5. 
Thus $\mu_\text{va}(\Pc_{G'})=\mu_\text{Hilb}(\Pc_{G'})=4$. 
Moreover, one has $\mu_\text{hole}(\Pc_{G'})=13$. In addition, similar to the case of the above $G$, 
neither $4\Pc_{G'}$ nor $5\Pc_{G'}$ has (IDP). 
However, we can check that $k\Pc_{G'}$ has (IDP) for $k \geq 6$, 
implying $\mu_\text{midp}(\Pc_{G'})=\mu_\text{idp}(\Pc_{G'})=6$. 
\end{Example}

\end{document}